\documentclass[a4paper,12pt]{article}
\usepackage[utf8]{inputenc}
\usepackage{amsmath,amssymb,amsthm}
\usepackage{geometry}
\usepackage{fancyhdr}
\usepackage{titlesec}
\usepackage{xcolor}
\usepackage{tikz-cd}
\usepackage{stmaryrd}
\usepackage{tikz}
\usepackage{comment}
\usepackage{hyperref}
\usepackage{cleveref}
\usetikzlibrary{arrows.meta,positioning,calc}
\titleformat{\section}{\Large\bfseries\color{black}}{\thesection}{1em}{}
\titleformat{\subsection}{\large\bfseries\color{black}}{\thesubsection}{1em}{}
\titleformat{\subsubsection}{\normalsize\bfseries\color{black}}{\thesubsubsection}{1em}{}

\theoremstyle{plain}
\newtheorem{theorem}{Theorem}[section]
\newtheorem{proposition}[theorem]{Proposition}
\newtheorem{lemma}[theorem]{Lemma}
\newtheorem{corollary}[theorem]{Corollary}
\newtheorem{conjecture}[theorem]{Conjecture}

\theoremstyle{definition}
\newtheorem{definition}[theorem]{Definition}
\newtheorem{example}[theorem]{Example}

\theoremstyle{remark}
\newtheorem{remark}[theorem]{Remark}

\newcommand{\R}{\mathbb{R}}

\newcommand{\supp}{\text{supp}}
\newcommand{\Cliff}{\text{Cliff}}

\newcommand{\prop}{\hbox{prop}}
\newcommand{\Pen}{\hbox{Pen}}
\newcommand{\Ind}{\text{Ind}}
\newcommand{\ev}{\text{ev}}

\begin{document}

\begin{center}
    \Large\textbf{The Coarse Novikov Conjecture for Finite Products of Fibred Coarsely Embeddable Spaces}

    \vspace{0.5cm}
    \normalsize Liang Guo, Zheng Luo, Qin Wang

    \vspace{0.5cm}
    \today

\end{center}
\begin{abstract}
In this paper, we prove that the coarse Novikov conjecture holds for finite products of bounded-geometry proper metric spaces whose factors admit fibred coarse embeddings into possibly different target spaces.
A key ingredient is the construction of suitable coarsely proper algebras adapted to product spaces. 
Our other main tool is an iterated approach to relative higher index theory. By constructing the corresponding multi-stage Roe algebras and index maps, we establish a reduction theorem that bridges these iterated relative statements with the global coarse Novikov conjecture.
\end{abstract}

\section{Introduction}

The coarse Novikov conjecture provides a large-scale index
theoretic framework linking the geometry of a proper metric space to the $K$-theory of its Roe algebra. For metric spaces satisfying various geometric conditions, the coarse Novikov conjecture has been established in a number of important cases \cite{Yu98,Yu00,KY12,CWY13}.

However, it is still unknown in general whether the coarse Novikov conjecture is preserved under finite products.
In \cite{DG24,Zhang25}, the authors studied the coarse Baum--Connes conjecture for product spaces.
Although these works obtained positive results, the versions of the coarse Baum--Connes conjecture considered there are all formulated with coefficients.

Under the weaker assumption of admitting a fibred coarse embedding, \cite{GWZ25} proves that, in the sparse setting, i.e.\ for coarse disjoint unions of finite metric spaces, the coarse Novikov conjecture holds for finite products of sparse spaces. More precisely, if \(X_1,\dots,X_N\) are \emph{sparse} spaces admitting fibred coarse embeddings into suitable model targets, then the coarse Novikov conjecture holds for the product
\[
\prod_{i=1}^N X_i.
\]
However, products of sparse spaces form only a rather restricted class of product spaces; many natural examples arising in coarse geometry lie outside the sparse framework.

In this paper, we remove the sparsity assumption and allow the factors to admit fibred coarse embeddings into possibly different target spaces. In this more general setting, we obtain the following result.

\begin{theorem}\label{mainthm}
Let $N\ge 2$, and let $X_1,\dots,X_N$ be proper metric spaces with bounded geometry. 
Assume that for each $i$ the space $X_i$ admits a fibred coarse embedding into one of the following model targets:
a Hilbert space, a Hadamard manifold, or an $\ell^p$-space with $1\le p<\infty$.
Then the \textbf{coarse Novikov conjecture holds for the finite product}
\[
X\ :=\ \prod_{i=1}^N X_i .
\]
\end{theorem}

This theorem provides a large new supply of examples of spaces satisfying the coarse Novikov conjecture, including in particular finite products of warped cones under natural geometric and analytic assumptions. In this sense, our results provide new progress toward understanding the coarse Novikov conjecture for finite products.

As an application, we have the following corollary.
\begin{corollary}
Let \(N,M\ge 1\).

For \(i=1,\dots,N\), let \(\Gamma_i\) be a finitely generated residually finite hyperbolic group, and
let \(\{\Gamma_{i,n}\}_{n\in\mathbb N}\) be a filtration of \(\Gamma_i\).
Write \(\Box_{\{\Gamma_{i,n}\}}\Gamma_i\) for the associated box space.

For \(j=1,\dots,M\), let \(\Lambda_j\) be a finitely generated discrete group acting on a compact manifold \(M_j\),
and assume that
\begin{enumerate}
  \item the action \(\Lambda_j\curvearrowright M_j\) is linearisable in an \(\ell^{p_j}\)-space and is free
  (or \(M_j\) contains a dense free orbit);
  \item \(\Lambda_j\) admits a proper affine isometric action on an \(\ell^{p_j}\)-space,
\end{enumerate}
for some \(1\le p_j<\infty\).
Let \(\mathcal O_{\Lambda_j}(M_j)\) be the corresponding warped cone.

Then the coarse Novikov conjecture holds for the mixed finite product
\[
X\ :=\ \Big(\prod_{i=1}^N \Box_{\{\Gamma_{i,n}\}}\Gamma_i\Big)\ \times\
\Big(\prod_{j=1}^M \mathcal O_{\Lambda_j}(M_j)\Big).
\]
\end{corollary}

To establish our main result, our starting point is the \emph{relative higher index theory} for a pair $(X, Y)$, where $Y\subseteq X$, recently introduced in \cite{GWZ25}. In their work, this framework was utilized to conduct a preliminary study of the coarse Novikov conjecture for product spaces. Specifically, they proved that if a collection of spaces all admit a fibred coarse embedding (FCE) into Hilbert space, their product space admits an FCE into a Hilbert space \emph{relative to its boundary} (we recall this notion in \Cref{fcebdry}).

To establish our main result in full generality, we extend this relative framework by removing two major restrictions. First, we allow the factors to admit fibred coarse embeddings into possibly different target spaces, rather than a single common Hilbert space. To do this, we introduce a new construction of a \emph{coarsely proper algebra} adapted to product spaces (see \Cref{exa: coarsely proper algebra}). Secondly, we drop the assumption that the underlying spaces are \emph{sparse}. In the sparse setting, the reduction process relies on reducing the dimension of the skeleton. However, for general proper metric spaces, such a skeleton-based dimension reduction fails.

To overcome this obstacle, we develop \emph{iterated higher index theory}. 
Rather than reducing the dimension of the skeleton, our approach proceeds by removing the faces of the product one at a time; see \Cref{thm:iterated-reduction} for the precise reduction statement.

An advantage of this face-by-face procedure is that it applies even when the intermediate pairs no longer admit relative fibred coarse embeddings. 
Thus, the iterated framework yields new examples of pairs satisfying the relative coarse Novikov conjecture beyond the scope of the relative FCE method.

\noindent\textbf{Organization.}
The paper is organized as follows. 
In Section~2, we recall Roe algebras and the coarse Baum--Connes and coarse Novikov conjectures.
In Section~3, we first recall the relative higher index theory. Then, for product spaces whose factors admit FCE into different target spaces, we construct a new coarsely proper algebra. Finally, we use this new algebra to prove that the associated twisted assembly map is an isomorphism.
In Section~4, we develop \emph{iterated higher index theory}, relate the finite iterated theory to the usual relative theory associated with the union of the closed subsets, and establish the face-by-face reduction used in the proof of the main result. We also introduce \emph{countably iterated ghostly ideals} and point out a phenomenon that does not occur in the finite case.
Finally, in Section~5, we establish the main theorem on the coarse Novikov conjecture for products in a more general setting, and derive further consequences and examples.

\paragraph{Acknowledgement.}
Z.~Luo is partially supported by the National Natural Science Foundation of China (No.~12401155).
Q.~Wang is partially supported by the National Natural Science Foundation of China (No.~12571135).

\section{Preliminaries}\label{sec:prelim}

We briefly review the basic notions of the coarse Baum--Connes conjecture and coarse Novikov conjecture in this part.
Throughout, \(X\) denotes a proper metric space, i.e.\ every closed bounded subset of \(X\) is compact.
Let \(C_0(X)\) be the \(C^*\)-algebra of complex-valued continuous functions on \(X\) vanishing at infinity.

\medskip
\noindent\textbf{\(X\)-modules.}
An \emph{\(X\)-module} is a separable Hilbert space \(H_X\) equipped with a faithful, non-degenerate
\(*\)-representation \(\pi:C_0(X)\to \mathcal B(H_X)\).
We say that \(H_X\) is \emph{ample} if \(\pi(f)\) is non-compact for every nonzero \(f\in C_0(X)\).
When no confusion arises, we write \(fh\) for \(\pi(f)h\).

\medskip
\noindent\textbf{Support, propagation, and local compactness.}
Let \(T\in \mathcal B(H_X)\). The \emph{support} of \(T\), denoted \(\supp(T)\subseteq X\times X\), is defined as the complement of the set of
\((x,y)\in X\times X\) for which there exist \(f,g\in C_0(X)\) with \(f(x)\neq 0\) and \(g(y)\neq 0\) such that
\[
gTf=0.
\]
The \emph{propagation} of \(T\) is
\[
\prop(T)\ :=\ \sup\{\, d(x,y)\ |\ (x,y)\in \supp(T)\,\}\ \in [0,\infty],
\]
and \(T\) has \emph{finite propagation} if \(\prop(T)<\infty\).
We say that \(T\) is \emph{locally compact} if \(fT\) and \(Tf\) are compact operators for all \(f\in C_0(X)\).

\medskip
\noindent\textbf{Bounded geometry.}
Assume that $X$ is discrete. We say $X$ has \emph{bounded geometry} if for every \(R>0\) there exists \(N_R<\infty\) with
\(\sup_{x\in X}|B(x,R)|\le N_R\).

\begin{definition}[Roe algebra]\label{def:roe-algebra}
Let \(X\) be a proper metric space and let \(H_X\) be an (ample) \(X\)-module.
Let \(\mathbb C[X]\) denote the \(*\)-algebra of all locally compact finite-propagation operators on \(H_X\).
The \emph{(reduced) Roe algebra} of \(X\) is the operator-norm completion
\[
C^*(X)\ :=\ \overline{\mathbb C[X]}^{\ \|\cdot\|}\ \subseteq\ \mathcal B(H_X).
\]
\end{definition}

\begin{definition}[Localization algebra]\label{def:localization-algebra}
Let \(X\) and \(H_X\) be as above.
The \emph{localization algebra} \(C_L^*(X)\) is the operator-norm completion of the \(*\)-algebra of all
bounded, uniformly norm-continuous functions
\[
g:[0,\infty)\longrightarrow C^*(X)
\]
such that \(g(t)\) has finite propagation for each \(t\ge 0\) and
\[
\prop\bigl(g(t)\bigr)\longrightarrow 0 \qquad \text{as } t\to\infty.
\]
We equip this \(*\)-algebra with the norm
\[
\|g\|\ :=\ \sup_{t\ge 0}\|g(t)\|.
\]
Moreover, there is a canonical evaluation \(*\)-homomorphism at \(0\),
\[
\ev_0:\ C_L^*(X)\to C^*(X),\qquad \ev_0(g)=g(0).
\]
\end{definition}

For proper metric spaces, the $K$-theory of $C_L^*(X)$ models the $K$-homology
of $X$, and the evaluation map induces the (coarse) assembly map
\[
\mu\ :=\ (\ev_0)_*:\lim_{d\to\infty}\ K_*\bigl(C_L^*(P_d(X))\bigr)\ \longrightarrow\ K_*\!\bigl(C^*(X)\bigr).
\]

\begin{definition}[Coarse Baum--Connes and coarse Novikov conjectures]\label{def:cbc-cnc}
Let $X$ be a bounded geometry proper metric space.
\begin{itemize}
  \item The \emph{coarse Baum--Connes conjecture} (CBC) for $X$ asserts that the assembly map
  $\mu:\lim_{d\to\infty}K_*(C_L^*(P_d(X)))\to K_*(C^*(X))$ is an isomorphism.
  \item The \emph{coarse Novikov conjecture} (CNC) for $X$ asserts that $\mu$ is injective.
\end{itemize}
\end{definition}

We do not review the coarse-geometric background here; the reader may consult \cite{NY12,Roe03} for the relevant material.

\section{Relative fibred coarse embedding and relative higher index theory}
In this section, we extend the framework of \cite{GWZ25} to the setting of multiple target spaces, and at the same time generalize the sparse setting considered there to arbitrary bounded-geometry proper metric spaces.

\subsection{Relative fibred coarse embedding}
The following definition, originating from \cite{GWZ25}, generalizes the notion of fibred coarse embedding introduced in \cite{CWY13}. In \cite{GWZ25} an additional sparseness assumption is imposed on the underlying space.
We will show that this assumption is in fact redundant, and the same conclusions remain valid
in the general proper bounded-geometry setting.
\begin{definition}[Fibred coarse embedding relative to a subspace]\label{def:relative-fce}
Let $M$ be a metric space serving as a model space (for example, a Hilbert space, a Hadamard manifold, or an $\ell^p$-space with $1 \le p < \infty$).  
A metric space $X$ is said to \emph{admit a fibred coarse embedding into $M$ relative to a subspace $Y \subseteq X$} if there exist:
\begin{itemize}
    \item a field of model spaces $\{M_x\}_{x \in X}$, each isometric to $M$;
    \item a section $s: X \to \bigsqcup_{x \in X} M_x$ (that is, $s(x) \in M_x$ for every $x \in X$);
    \item two non-decreasing control functions $\rho_{-}, \rho_{+}: \mathbb{R}_{+} \to \mathbb{R}_{+}$ satisfying $\lim_{r \to \infty} \rho_{\pm}(r) = \infty$;
\end{itemize}
such that the following condition holds:

For every $R > 0$, there exists $S > 0$ with the property that for each point $x \in X \setminus \operatorname{Pen}(Y,S)$, one can find a \emph{local trivialization}
\[
t_{x,R}: \{M_z\}_{z \in B(x,R)} \longrightarrow B(x,R) \times M
\]
satisfying:
\begin{enumerate}
    \item[\textnormal{(1)}] \textbf{(Controlled distortion)}:  
    for all $z_1, z_2 \in B(x,R)$,
    \[
    \rho_{-}\!\big(d(z_1,z_2)\big)
    \;\leq\;
    d_M\!\big(t_{x,R}(z_1)(s(z_1)),\, t_{x,R}(z_2)(s(z_2))\big)
    \;\leq\;
    \rho_{+}\!\big(d(z_1,z_2)\big);
    \]
    \item[\textnormal{(2)}] \textbf{(Compatibility on overlaps)}:  
    for any $x,y \in X \setminus Y_S$ with $B(x,R) \cap B(y,R) \neq \varnothing$, there exists an isometry $t_{xy,R}: M \to M$ such that
    \[
    t_{x,R}(z) \circ t_{y,R}^{-1}(z) = t_{xy,R}, \quad \forall z \in B(x,R) \cap B(y,R).
    \]
\end{enumerate}
\end{definition}

\begin{remark}
If the distinguished subset $Y$ is bounded, then a fibred coarse embedding of $X$ into $M$ relative to $Y$ coincides with the usual notion of a fibred coarse embedding of $X$ into $M$.  
Hence, this relative version serves as a natural generalization of the classical concept.  
\end{remark}

Let $\{X_i\}_{i=1}^N$ be a finite collection of proper metric spaces, each endowed with a distinguished base point $o_i \in X_i$.  
Equip the product space $X = \prod_{i=1}^N X_i$ with the $\ell^2$-metric.

For every $M >0$ and $1 \leq i \leq N$, we define the \emph{$i$-th thickened face} of $X$ by
\[
F_M^{(i)} = 
\Big( \prod_{j \neq i} X_j \Big) \times B_{X_i}(o_i, M),
\]
where $B_{X_i}(o_i, M) = \{ x \in X_i \mid d_{X_i}(x, o_i) \le M \}$ denotes the closed $M$-neighbourhood of $o_i$ in $X_i$.

The \emph{$M$-boundary} of $X$ is then defined as the union of all these thickened faces:
\[
F_M = \bigcup_{i=1}^N F_M^{(i)}.
\]

Intuitively, $F_M$ consists of those points in $X$ whose $i$-th coordinate lies within distance $M$ from the chosen base point $o_i$ for at least one index $i$.

\begin{example}
Let $X_1=X_2=\mathbb{R}$ with the standard metric and the base points $o_1=o_2=0$. 
Consider the product space
\[
X=X_1\times X_2=\mathbb{R}^2
\]
with the $\ell^2$-product metric.  

For each $M\in \mathbb{N}$, the first face with thickness $M$ is
\[
F_M^{(1)}=B_{X_1}(0,M)\times X_2=[-M,M]\times \mathbb{R},
\]
and the second face with thickness $M$ is
\[
F_M^{(2)}=X_1\times B_{X_2}(0,M)=\mathbb{R}\times [-M,M].
\]

Hence the $M$-boundary of $X$ is
\[
F_M = F_M^{(1)}\cup F_M^{(2)}
= \big( [-M,M]\times \mathbb{R}\big)\;\cup\;\big(\mathbb{R}\times [-M,M]\big).
\]

Geometrically, $F_M$ consists of the vertical strip of width $2M$ around the $y$-axis together with the horizontal strip of width $2M$ around the $x$-axis. 
In particular, for $M=1$, $F_1$ is the union of a vertical strip of width $2$ and a horizontal strip of width $2$ crossing through the origin.
\end{example}

\begin{figure}[htbp]
  \centering
  \begin{tikzpicture}[scale=0.5]
    \fill[blue!40,opacity=0.6] (-1,-5) rectangle (1,5);
    \fill[green!40,opacity=0.6] (-5,-1) rectangle (5,1);

    \draw[->] (-5,0) -- (5.2,0) node[right] {$x$};
    \draw[->] (0,-5) -- (0,5.2) node[above] {$y$};

    \draw[step=1cm, help lines, dotted] (-5,-5) grid (5,5);

    \node at (0.6,4.2) {$[-1,1]\times\mathbb{R}$};
    \node at (4.0,0.6) {$\mathbb{R}\times[-1,1]$};
    \node at (2.5,2.5) {$F_1$};
  \end{tikzpicture}
  \caption{The 1-boundary $F_1$ in $\mathbb{R}^2$. }
\end{figure}
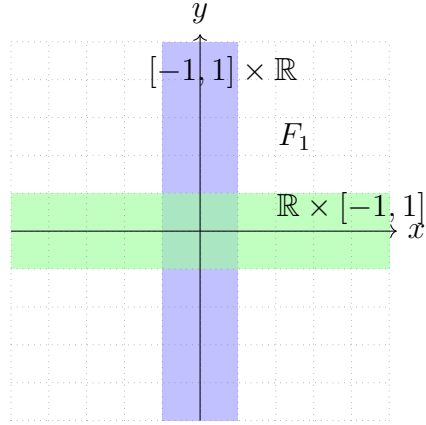

The family $\{F_M\}_{M>0}$ provides a coarse notion of ``boundary'' for the product space $X$,
we will see that if each factor admits a fibred coarse embedding,
then the product space inherits a relative version of this property with respect to its $1$-boundary.

\begin{proposition}\label{fcebdry}
Let \(N\ge 1\). For each \(i=1,\dots,N\), let \(X_i\) be a bounded-geometry proper metric space, and fix a base point \(o_i\in X_i\).
Assume that each \(X_i\) admits a fibred coarse embedding into a model target space \(B_i\),
where each \(B_i\) is either an \(\ell^{p_i}\)-space or a Hadamard space.
Let
\[
X=\prod_{i=1}^N X_i
\]
be equipped with the \(\ell^2\)-product metric, and let \(F_1\) denote its \(1\)-boundary.

Then \(X\) admits a fibred coarse embedding into the product target
\[
B:=\prod_{i=1}^N B_i
\]
relative to \(F_1\), where \(B\) is equipped with the \(\ell^2\)-product metric.
\end{proposition}

\begin{proof}[Sketch of proof]
For each \(i=1,\dots,N\), let \((B_{x_i})_{x_i\in X_i}\) be the field of target spaces over \(X_i\),
with each \(B_{x_i}\) isometric to \(B_i\), and let
\[
s_i:X_i\to \bigsqcup_{x_i\in X_i} B_{x_i}
\]
be the corresponding section. Assume all factorwise fibred coarse embeddings have the same control
functions \(\rho_\pm\).

Set
\[
B:=\prod_{i=1}^N B_i
\]
with the \(\ell^2\)-product metric. For \(x=(x_1,\dots,x_N)\in X=\prod_{i=1}^N X_i\), define
\[
\widetilde B_x:=\prod_{i=1}^N B_{x_i},
\qquad
s(x):=(s_1(x_1),\dots,s_N(x_N)).
\]

For each \(R>0\), choose \(M_i>0\) so that for every \(x_i\in X_i\setminus X_{i,M_i}\) there exists a
trivialization
\[
t_{x_i,R}:(B_z)_{z\in B(x_i,R)}\longrightarrow B(x_i,R)\times B_i.
\]
Let
\[
M:=\max\{M_i:1\le i\le N\}.
\]
Then \(F_M\) is a bounded neighbourhood of \(F_1\). For \(x=(x_1,\dots,x_N)\in X\setminus F_M\), define
the product trivialization
\[
t_{x,R}:=\prod_{i=1}^N t_{x_i,R}.
\]
Since
\[
B_X(x,R)\subseteq \prod_{i=1}^N B_{X_i}(x_i,R),
\]
this gives a trivialization over \(B_X(x,R)\).

Now for \(z,z'\in B_X(x,R)\), the \(\ell^2\)-product metric and the common control functions give
\[
d_B\bigl(t_{x,R}(z)(s(z)),\,t_{x,R}(z')(s(z'))\bigr)^2
=
\sum_{i=1}^N
d_{B_i}\bigl(t_{x_i,R}(z_i)(s_i(z_i)),\,t_{x_i,R}(z_i')(s_i(z_i'))\bigr)^2
\le
N\,\rho_+^2\bigl(d_X(z,z')\bigr),
\]
while for some \(i\) one has \(d_{X_i}(z_i,z_i')\ge d_X(z,z')/\sqrt N\), hence
\[
d_B\bigl(t_{x,R}(z)(s(z)),\,t_{x,R}(z')(s(z'))\bigr)
\ge
\rho_-\!\left(\frac{d_X(z,z')}{\sqrt N}\right).
\]
Thus the product embedding has control functions
\[
\widetilde\rho_-(t):=\rho_-\!\left(\frac{t}{\sqrt N}\right),
\qquad
\widetilde\rho_+(t):=\sqrt N\,\rho_+(t).
\]

Finally, the overlap condition is inherited coordinatewise, since the transition maps in each factor are
isometries of \(B_i\), and therefore their product is an isometry of \(B\).
Hence these product data define a fibred coarse embedding of \(X\) into \(B\) relative to \(F_1\).
\end{proof}

\begin{remark}\label{mono}
Suppose $F\subseteq F'\subseteq X$ are closed subsets. 
If $X$ admits a fibred coarse embedding into a Hilbert space relative to $F$, 
then it also admits a fibred coarse embedding (with the same control functions) relative to $F'$. 
Indeed, for every $R>0$ one has $X\setminus \Pen(F',R)\subseteq X\setminus \Pen(F,R)$, 
so the required local trivializations for $X\setminus \Pen(F,R)$ restrict to
$X\setminus \Pen(F',R)$. We refer to this phenomenon as the \em{monotonicity of the relative condition}.
\end{remark}

\subsection{Relative higher index theory}

In this subsection, we recall the relative Roe algebra and the associated
relative index map introduced in \cite{GWZ25}.
Let \(X\) be a proper metric space with bounded geometry, and let
\(Y\subseteq X\) be a closed subset.

For \(T\in C^*(X)\) and \(\varepsilon>0\), define the
\(\varepsilon\)-support of \(T\) by
\[
\supp_{\varepsilon}(T)
:=
\bigl\{(x,y)\in X\times X:\|T(x,y)\|\geq \varepsilon\bigr\}.
\]
For a subset \(E\subseteq X\times X\), let
\[
r(E):=\{x\in X:(x,y)\in E\text{ for some }y\in X\}
\]
denote its projection onto the first coordinate.

\begin{definition}\label{def:ghostly-ideal}
The \emph{ghostly ideal} of \(C^*(X)\) associated with \(Y\) is defined by
\[
I_G(Y)
:=
\left\{
T\in C^*(X):
\begin{array}{l}
\text{for every }\varepsilon>0,\text{ there exists }R>0\\
\text{such that }
r\bigl(\supp_{\varepsilon}(T)\bigr)
\subseteq \Pen(Y,R)
\end{array}
\right\}.
\]
\end{definition}

\begin{definition}\label{def:relative-roe}
The \emph{relative Roe algebra at infinity} of \(X\) with respect to \(Y\)
is defined by
\[
C^*_{\infty,Y}(X)
:=
C^*(X)\big/I_G(Y).
\]
Thus, there is a short exact sequence
\[
0
\longrightarrow I_G(Y)
\longrightarrow C^*(X)
\xrightarrow{\;\pi_Y\;}
C^*_{\infty,Y}(X)
\longrightarrow 0.
\]
\end{definition}

For \(T\in C^*_{\infty,Y}(X)\), its propagation is defined by
\[
\prop(T)
:=
\inf\bigl\{
\prop(S):
S\in C^*(X),\ \pi_Y(S)=T
\bigr\}.
\]

We next introduce the localization-algebra counterpart of the relative
Roe algebra. Define
\[
C^*_{L,Y}\bigl(X;I_G(Y)\bigr)
:=
\left\{
f\in C_L^*(X):
f(t)\in I_G(Y)\text{ for every }t\geq 0
\right\}.
\]
This is a closed ideal of \(C_L^*(X)\).

\begin{definition}\label{def:relative-localization}
Let \(X\) be a proper metric space with bounded geometry, and let
\(Y\subseteq X\) be a closed subset. The \emph{relative localization
algebra at infinity} of \(X\) with respect to \(Y\), denoted by
\(C^*_{L,\infty,Y}(X)\), is the \(C^*\)-algebra consisting of all bounded,
uniformly norm-continuous functions
\[
f:[0,\infty)\longrightarrow C^*_{\infty,Y}(X)
\]
such that \(f(t)\) has finite propagation for every \(t\geq 0\) and
\[
\prop\bigl(f(t)\bigr)\longrightarrow 0
\qquad\text{as }t\to\infty.
\]
It is equipped with the supremum norm
\[
\|f\|:=\sup_{t\geq 0}\|f(t)\|.
\]
\end{definition}

Evaluation at \(t=0\) defines a canonical \(*\)-homomorphism
\[
\ev_0^Y:
C^*_{L,\infty,Y}(X)
\longrightarrow
C^*_{\infty,Y}(X),
\qquad
\ev_0^Y(f)=f(0).
\]
Consequently, it induces the relative index map
\[
\Ind_{Y,\infty}
:=
(\ev_0^Y)_*:
K_*\bigl(C^*_{L,\infty,Y}(X)\bigr)
\longrightarrow
K_*\bigl(C^*_{\infty,Y}(X)\bigr).
\]

\begin{definition}\label{def:relative-index-map}
The homomorphism induced by evaluation at zero,
\[
\Ind_{Y,\infty}
:=
(\ev_0^Y)_*:
K_*\bigl(C^*_{L,\infty,Y}(X)\bigr)
\longrightarrow
K_*\bigl(C^*_{\infty,Y}(X)\bigr),
\]
is called the \emph{relative index map} of \(X\) with respect to \(Y\).
\end{definition}

Accordingly, we define the relative \(K\)-homology at infinity of the
pair \((X,Y)\) by
\[
K_*^\infty(X;Y)
:=
K_*\bigl(C^*_{L,\infty,Y}(X)\bigr).
\]
With this notation, the relative index map takes the form
\[
\Ind_{Y,\infty}:
K_*^\infty(X;Y)
\longrightarrow
K_*\bigl(C^*_{\infty,Y}(X)\bigr).
\]

Now let \(P_d(X)\) denote the Rips complex of \(X\) at scale \(d\).
For every \(d\geq 0\), evaluation at zero gives a relative index map
\[
\Ind_{Y,\infty}^{d}:
K_*\bigl(C^*_{L,\infty,Y}(P_d(X))\bigr)
\longrightarrow
K_*\bigl(C^*_{\infty,Y}(P_d(X))\bigr).
\]
Using the canonical coarse equivalence between \(P_d(X)\) and \(X\),
the target may be identified with
\[
K_*\bigl(C^*_{\infty,Y}(X)\bigr).
\]
Passing to the direct limit over \(d\), we obtain the relative coarse
assembly map
\[
\mu_{Y,\infty}:
\varinjlim_{d\to\infty}
K_*\bigl(C^*_{L,\infty,Y}(P_d(X))\bigr)
\longrightarrow
K_*\bigl(C^*_{\infty,Y}(X)\bigr).
\]

\begin{conjecture}[Relative coarse Baum--Connes and Novikov conjectures]
\label{conj:relative-cbc-cnc}
Let \(X\) be a proper metric space with bounded geometry and let
\(Y\subseteq X\) be a closed subset.
\begin{enumerate}
\item The \emph{relative coarse Baum--Connes conjecture} for \((X,Y)\)
asserts that
\[
\mu_{Y,\infty}
\]
is an isomorphism.

\item The \emph{relative coarse Novikov conjecture} for \((X,Y)\)
asserts that
\[
\mu_{Y,\infty}
\]
is injective.
\end{enumerate}
\end{conjecture}

Following the approach of \cite{GLWZ23,GWZ25}, we introduce the notion of coarsely proper coefficient
algebras and develop the twisted assembly map with coefficients in such algebras for general
bounded geometry metric spaces, not only for spaces that are sparse relative to a subset $Y$.
Within this broader framework, we establish the relative coarse Novikov conjecture in certain cases. 

As noted earlier, the relative index map considered here is precisely the \(k=1\) case of the previously defined iterated index map. Likewise, the relative Roe algebra, relative localization algebra, and relative \(K\)-homology appearing here are exactly the corresponding \(k=1\) instances of their iterated counterparts.

\begin{definition}\label{NCalg}
Let $\mathcal{A}$ be an $N$-$C^*$-algebra.  
We say that $\mathcal{A}$ is a \emph{coarsely proper algebra associated with the relative fibred coarse embedding} of $X$ into $M$ with respect to $Y$,  
if the following data exist:

\begin{enumerate}
  \item A field of $C^*$-algebras $\left(\mathcal{A}_x\right)_{x \in X}$ over $X$ such that each $\mathcal{A}_x$ is isomorphic to $\mathcal{A}$;

  \item For each $x \in X$, a family of open subsets $\{O_{x,R}\}_{R \ge 0}$ of $N_x$ forming an increasing nest whose union $\bigcup_{R \ge 0} O_{x,R}$ is dense in $N_x$,  
  where $\mathcal{A}_x$ is regarded as an $N_x$-$C^*$-algebra.
\end{enumerate}

Moreover, for every $k \in \mathbb{N}$ and $x \in X \setminus Y_k$,  
the trivialization associated with the fibred coarse embedding induces a corresponding trivialization of the field of $C^*$-algebras.  

\[
t_x : \bigcup_{z \in B(x, l_k)} \mathcal{A}_z \longrightarrow B(x, l_k) \times \mathcal{A}.
\]

This trivialization satisfies the following conditions:

\begin{enumerate}
 \item[(1)]
For each \(z\in B(x,R)\), the map
$t_{x,R}(z):\mathcal A_z\longrightarrow\mathcal A$
is a \(C^*\)-isomorphism inducing a homeomorphism
\[
f_{t_{x,R}(z)}:N_z\longrightarrow N.
\]
There exist non-decreasing proper functions
$\rho_-,\rho_+:\mathbb R_+\longrightarrow\mathbb R_+$
such that, for all \(z,z'\in B(x,R)\) and \(r>0\),
\[
f_{t_{x,R}(z)}(O_{z,\rho_-(r)})
\cap
f_{t_{x,R}(z')}(O_{z',\rho_-(r)})
=
\varnothing
\]
whenever
\[
d_M\!\left(
t_{x,R}(z)(s(z)),
t_{x,R}(z')(s(z'))
\right)
\ge 2r,
\]
and
\[
f_{t_{x,R}(z)}(O_{z,r})
\subseteq
f_{t_{x,R}(z')}
\left(
O_{z',\,r+\rho_+(d(z,z'))}
\right).
\]

\item[(2)]
For any \(x,x'\in X\setminus\Pen(Y,S)\) with
\[
B(x,R)\cap B(x',R)\neq\varnothing,
\]
there exists a \(C^*\)-isomorphism
$t_{xx',R}:\mathcal A\longrightarrow\mathcal A$
such that
\[
t_{x,R}(z)\circ t_{x',R}^{-1}(z)
=
t_{xx',R}
\]
for every
\[
z\in B(x,R)\cap B(x',R).
\]
\end{enumerate}
\end{definition}

\begin{example}\label{exa: coarsely proper algebra}
If $X$ admits a \emph{relative fibred coarse embedding} into an $\ell^{p}$-space $B$ with respect to a fixed subspace $Y$, 
then, according to~\cite{GLWZ24}, one can associate to it a \emph{coarsely proper algebra} 
$\mathcal{A}(B)$ corresponding to the relative fibred coarse embedding. 

Similarly, if $X$ admits a relative fibred coarse embedding into a Hadamard manifold $M$ 
with respect to a subspace $Y$, then the algebra
\[
\mathcal{A}(M):=C_{0}\!\bigl(M,\mathrm{Cliff}_{\mathbb{C}}(TM)\bigr)
\]
can likewise be shown to be a \emph{coarsely proper algebra} associated with the relative fibred coarse embedding.

More generally, if the target space is a finite product
\[
E=B^{(1)}\times\cdots\times B^{(m)}\times M,
\]
where each $B^{(j)}$ is an $\ell^{p_j}$-space and $M$ is a Hadamard manifold, then one associates to $E$ the graded tensor product algebra
\[
\mathcal A(E):=
\mathcal A(B^{(1)})\widehat\otimes\cdots\widehat\otimes \mathcal A(B^{(m)})
\widehat\otimes \mathcal A(M).
\]
This algebra is again a coarsely proper algebra associated with the corresponding relative fibred coarse embedding.

The full details are provided in \cite{GLWZ23} and are omitted here for brevity.
\end{example}

\begin{definition}  
Let $X$ be a proper metric space and let $Y \subseteq X$ be a closed subspace.  
The \emph{algebraic twisted Roe algebra at infinity relative to $Y$}, denoted by  
$\mathbb{C}_{\infty, Y}[X, \mathcal{A}]$,  
is defined as the set of equivalence classes $[T]$ of families of operators $T = (T(x,y))_{x,y \in X}$ satisfying the following conditions:  
\begin{enumerate}  
\item
There exists $r > 0$ such that $T(x,y) = 0$ whenever $d(x,y) > r$.  

\item 
There exists $R > 0$ such that  
$\operatorname{supp}_N(T(x,y)) \subseteq f_{t_x(x)}(O_{x,R})$  
for all $x,y \in X_n$,  
where $f_{t_x(x)}$ is defined as in Definition \ref{NCalg} and $T(x,y)$ is viewed as an operator on $\mathcal{H}_{P_d(X), \mathcal{A}}$.  
\end{enumerate}  

Two elements $T$ and $S$ are identified if  
\[
T \sim S  
\quad \Longleftrightarrow \quad  
\lim_{R \to \infty} \sup_{x,y \in X \setminus Y_R} \|T(x,y) - S(x,y)\| = 0,
\]
where $Y_R$ denotes the $R$–neighborhood of $Y$.  
\end{definition}

Following the standard construction of twisted Roe algebras (cf.~\cite{DGWY25,Yu00}),  
for each $C^*$-isomorphism $t_{xy}:\mathcal{A}\to\mathcal{A}$ we define the induced isometry on $\mathcal{H}_{P_d(X),\mathcal{A}}$ by  
\[
t_{xy}(\delta_z\otimes\delta_n\otimes a)=\delta_z\otimes\delta_n\otimes t_{xy}(a),
\]
and the corresponding $*$-automorphism 
$\mathrm{Ad}_{t_{xy}}(T)=t_{xy}Tt_{xy}^{-1}$ on $\mathcal{B}(\mathcal{H}_{P_d(X),\mathcal{A}})$.

The algebraic operations on $\mathbb{C}_{Y,\infty}[(P_d(X),\mathcal{A})]$ are then defined by
$$
(TS)(x,y)=\sum_{z\in X} T(x,z)\,\mathrm{Ad}_{t_{xz}}(S(z,y)),$$
$$T^*(x,y)=\mathrm{Ad}_{t_{xy}}(T(y,x)).$$

These formulas yield a well-defined $*$-algebra structure,  
since the support control relation 
$f_{t_{xy}}(f_{t_y(y)}(O_{y,R}))\subseteq f_{t_x(x)}(O_{x,R+\rho(d(x,y))})$
ensures closure under multiplication and adjoint (see~\cite{DGWY25}).

The representation of this algebra on 
$E=\bigoplus_{z\in X}\mathcal{H}_{P_d(X),\mathcal{A}}$  
is given componentwise by
\[
(T\tilde{\xi})_z^y=\sum_{x\in X}\mathrm{Ad}_{t_{zx}}(T(x,y))\,\tilde{\xi}_x^y,
\]
which fits the algebraic structure as in~\cite[Remark~3.5]{DGWY25}.For details of this construction, see~\cite{DGWY25,GWZ25}.

Finally, the ghostly ideal associated with $Y$ is defined by
\[
\mathcal{I}_G(Y,\mathcal{A})=
\{\,T\in C^*(P_d(X),\mathcal{A})\mid 
\lim_{k\to\infty}\sup_{x,y\in X\setminus Y_k}\|T(x,y)\|\le\varepsilon\,\}.
\]

\begin{definition}
The norm on  
$\mathbb{C}_{Y, \infty}\!\left[\left(P_d(X), \mathcal{A}\right)\right]$  
is defined by  
\[
\|[T]\|
=
\inf\bigl\{
\|T+S\|:
S\in\mathcal I_G(Y,\mathcal A)
\bigr\}.
\]
The completion of  
$\mathbb{C}_{Y, \infty}\!\left[\left(P_d(X), \mathcal{A}\right)\right]$  
under this norm is defined to be the twisted Roe algebra at infinity relative to $Y$, denoted by  
$C_{Y, \infty}^*\!\left(P_d(X), \mathcal{A}\right)$.
\end{definition}

\begin{definition} For each $d \geq 0$, the algebraic twisted localization algebra at infinity of $X$ relative to $Y$, denoted by $\mathbb{C}_{L, Y, \infty}\left[\left(P_d\left(X\right), \mathcal{A}\right)\right]$, consists of all bounded, uniformly continuous functions
$$
g: \mathbb{R}_{+} \rightarrow \mathbb{C}_{Y, \infty}\left[\left(P_d\left(X\right), \mathcal{A}\right)\right]
$$
such that $g(t)$ satisfies the following conditions
\begin{enumerate}
\item $\operatorname{prop}\left(g(t)\right) \rightarrow 0$ as $t \rightarrow \infty$;
\item there exists $R>0$ such that $\operatorname{supp}\left(g(t)\right)(x, y) \subseteq f_{t_x(x)}\left(O_{x, R}\right)$ for all $t \in \mathbb{R}_{+},$ and $x, y \in X$.
\end{enumerate}
The twisted localization algebra at infinity relative to $Y$ is defined to be the $C^*$-completion of $\mathbb{C}_{L, Y, \infty}[\left(P_d\left(X\right), \mathcal{A}\right]$ under the norm,
$$
\|g\|_{\text {red }}=\sup _{t \in \mathbb{R}_{+}}\|g(t)\|
$$
denoted by $C_{L, Y, \infty}^*\left(P_d\left(X\right), \mathcal{A}\right).$
\end{definition}

The evaluation maps $g \mapsto g(0)$ for both cases induces the following index maps
$$
\begin{gathered}
\operatorname{Ind}_{\mathcal{A}}: K_*\left(C_{L, Y, \infty}^*\left(P_d\left(X\right), \mathcal{A}\right)\right) \rightarrow K_*\left(C_{Y, \infty}^*\left(P_d\left(X\right), \mathcal{A}\right)\right).
\end{gathered}
$$

Passing $d$ to infinity, we then have the twisted relative assembly map at infinity.
$$
\mu_{\mathcal{A}}: \lim _{d \rightarrow \infty} K_*\left(C_{L, Y, \infty}^*\left(P_d\left(X\right), \mathcal{A}\right)\right) \rightarrow \lim _{d \rightarrow \infty} K_*\left(C_{Y, \infty}^*\left(P_d\left(X\right), \mathcal{A}\right)\right).
$$

We can, following the approach in \cite{GLWZ24,GWZ25}, use the cutting-and-pasting technique together with the Mayer–Vietoris sequence argument to prove the following theorem.
\begin{theorem}{\label{twist}}
If $\mathcal{A}$ is a coarsely proper algebra associated with the relative fibred coarse embedding of $X$ with respect to $Y$, then the twisted relative assembly map $\mu_{\mathcal{A}}$ is an isomorphism.
\end{theorem}

\subsection{Almost flat Bott generators.} 
In this subsection, we provide a conceptual description of almost flat Bott generators for coarsely proper algebras.
\begin{definition}[\cite{GLWZ23,GWZ25}]{\label{afbg}}
 A coarsely proper algebra $\mathcal{A}$ is said to admit a family of uniformly almost flat Bott generators if,
\begin{enumerate}
\item for any $x \in X$ and $\tau \geq 1$, there exists a generator $\left[b_{x, \tau}\right] \in K_*\left(\mathcal{A}_x\right)$ such that for any $\varepsilon>0$ and $r \geq 0$, there exists $T>1$ such that
$$
\left\|t_x(x)\left(b_{x, \tau}\right)-t_x(y)\left(b_{y, \tau}\right)\right\| \leq \varepsilon,
$$
for all $x, y$ with $d(x, y) \leq r$ and $x, y$ in the same trivialization, and $\tau \geq T$, where $\mathcal{A}_x$ and $t_x$ is defined as in Definition \ref{NCalg};
\item for any $\tau \geq 1$, there exists $R>0$ such that $\operatorname{supp}\left(b_{x, \tau}\right) \subseteq O_{x, R}$ for any $x \in X$, where $O_{x, R}$ is defined as in Definition \ref{NCalg}.
\end{enumerate}
\end{definition}

\begin{example}[$\ell^p$-spaces and Hadamard manifolds]\label{ex:lp-hadamard-bott}
We record the basic Bott homomorphisms for the two types of target spaces that will appear in the product construction.

\smallskip
\noindent{\bf (1) The $\ell^p$-case.}
Let $B_x=\ell^p(\mathbb N,\mathbb R)$ for each $x\in X$, and let
\[
H=\ell^2(\mathbb N,\mathbb R).
\]
For any $\tau\in \mathbb R_+$ and $x\in X$, define a graded $*$-homomorphism
\[
\beta_{x,\tau}^{(p)}:\mathcal S\longrightarrow \mathcal A(B_x)\widehat\otimes \mathcal K
\]
by
\[
\beta_{x,\tau}^{(p)}(f)=f_\tau(C_{s(x)})\otimes p,
\]
where $p$ is a rank-one projection in $\mathcal K$, $f_\tau(r)=f(r/\tau)$, and
\[
C:B_x\longrightarrow \Cliff_{\mathbb C}(H)
\]
is the Clifford generator introduced in \cite[Section 3]{GLWZ24}, defined via the $\frac p2$-H\"older extension of the Mazur map.

By \cite[Remark 2.9.13]{WY20}, each homomorphism $\beta_{x,\tau}^{(p)}$ determines a class
\[
[\beta_{x,\tau}^{(p)}]\in K_1\bigl(\mathcal A(B_x)\bigr).
\]
Moreover, by \cite[Theorem 3.15]{GLWZ24}, for every $x\in X$ and every $\tau\in \mathbb R_+$, the class$[\beta_{x,\tau}^{(p)}]$
is a generator of $K_1(\mathcal A(B_x))$ for each $x\in X$ and $\tau\in\R_+$.
Furthermore, the family
$\bigl\{[\beta_{x,\tau}^{(p)}]\bigr\}_{x\in X,\ \tau\ge 1}$
forms a family of uniformly almost flat Bott generators in the sense of Definition \ref{afbg}.
Hence it induces a Bott map on $K$-theory.
\[
\beta_*:
K_1\!\left(C_{Y,\infty}^*\!\bigl((P_d(X))\bigr)\right)
\longrightarrow
K_0\!\left(C_{Y,\infty}^*\!\bigl((P_d(X),\mathcal{A}(B_x))\bigr)\right),
\]

\smallskip
\noindent{\bf (2) The Hadamard case.}
Let $M$ be a Hadamard manifold. For each $(x,t)\in M\times \mathbb R_+$, write
\[
t\mathbb R=
\begin{cases}
\mathbb R, & t\neq 0,\\
\{0\}, & t=0.
\end{cases}
\]
Define
\[
\Pi(M):=\prod_{(x,t)\in M\times \mathbb R_+}\Cliff_{\mathbb C}(T_xM\oplus t\mathbb R),
\]
and
\[
\Pi_b(M):=\Bigl\{\sigma\in \Pi(M):\sup_{(x,t)\in M\times \mathbb R_+}\|\sigma(x,t)\|<\infty\Bigr\}.
\]
For $x_0\in M$, define the Clifford generator
\[
C_{x_0}(x,t):=\bigl(-\log_x(x_0),t\bigr)\in T_xM\oplus t\mathbb R
\subseteq \Cliff_{\mathbb C}(T_xM\oplus t\mathbb R).
\]

Write $\mathcal S=C_0(\mathbb R)$ and decompose it into even and odd parts:
$C_0(\mathbb R)_{\mathrm{ev}}$ and $C_0(\mathbb R)_{\mathrm{odd}}$.
The functional calculus for $\Cliff_{\mathbb C}(T_xM\oplus t\mathbb R)$ is defined as follows.
For $f\in C_0(\mathbb R)_{\mathrm{ev}}$ and $v\in T_xM\oplus t\mathbb R\subseteq \Cliff_{\mathbb C}(T_xM\oplus t\mathbb R)$, set
\[
f(v)=f(\|v\|)\in \mathbb C\subseteq \Cliff_{\mathbb C}(T_xM\oplus t\mathbb R).
\]
For $g\in C_0(\mathbb R)_{\mathrm{odd}}$, define
\[
g(v)=
\begin{cases}
0, & v=0,\\[0.5em]
g(\|v\|)\dfrac{v}{\|v\|}, & v\neq 0.
\end{cases}
\]
Using this functional calculus, one defines a graded $*$-homomorphism
\[
\beta^M_{x_0}:\mathcal S\longrightarrow \Pi_b(M)
\]
by
\[
(\beta^M_{x_0}(h))(x,t)=h(C_{x_0}(x,t)),
\qquad h\in \mathcal S.
\]
Let $\mathcal A(M)$ denote the $C^*$-subalgebra of $\Pi_b(M)$ generated by $\{\beta^M_{x_0}(h):x_0\in M,\ h\in \mathcal S\}$.

By \cite{GWY21}, for each $x\in M$, the Bott homomorphism $\beta^M_x$ determines a generator of $K_1(\mathcal A(M))$ for each $x\in X$.
Moreover, by Sections 5 and 7 of \cite{GWY21}, these Bott homomorphisms induce the corresponding Bott map on $K$-theory,
\[
\beta_*:
K_1\!\left(C_{Y,\infty}^*\!\bigl((P_d(X))\bigr)\right)
\longrightarrow
K_0\!\left(C_{Y,\infty}^*\!\bigl((P_d(X),\mathcal{A}(M))\bigr)\right),
\]

\smallskip
\noindent{\bf (3) The Product case.}
Let $E=B^{(1)}\times\cdots\times B^{(m)}\times M$, where each $B^{(j)}=\ell^{p_j}(\mathbb N,\mathbb R)$ and $M$ is a Hadamard manifold.
We define the associated proper algebra by
\[
\mathcal A(E):=
\mathcal A(B^{(1)})\otimes\cdots\otimes \mathcal A(B^{(m)})
\otimes \mathcal A(M).
\]

For each $x\in X$ and $\tau\in \mathbb R_+$, let $[b^{(j)}_{x,\tau}]\in K_1\bigl(\mathcal A(B^{(j)})\bigr), j=1,\dots,m$
be the Bott generators associated to the $\ell^{p_j}$-factors, and let
$[b^M_x]\in K_1\bigl(\mathcal A(M)\bigr)$
be the Bott generator associated to the Hadamard factor.
We then define the Bott generator for the product target $E$ by
\[
[b^E_{x,\tau}]
:=
[b^{(1)}_{x,\tau}]\otimes\cdots\otimes [b^{(m)}_{x,\tau}]
\otimes [b^M_x]
\in K_{m+1}\bigl(\mathcal A(E)\bigr).
\]
Equivalently, at the level of Bott homomorphisms, one sets
\[
\beta^E_{x,\tau}
:=
\beta^{(1)}_{x,\tau}\otimes\cdots\otimes \beta^{(m)}_{x,\tau}
\otimes \beta^M_x.
\]
By K\"unneth formula, $\beta^E_{x,\tau}$ is indeed a generator of $K_{m+1}\bigl(\mathcal A(E)\bigr)$, see \cite{Bla98}. It is straightforward to check that $\beta^E_{x,\tau}$ is a family of almost flat Bott generators.
Consequently, the family
$\{[b^E_{x,\tau}]\}_{x\in X,\ \tau\ge 1}$
induces the corresponding Bott map on $K$-theory
\[
\beta_*:
K_*\!\left(C_{Y,\infty}^*\!\bigl((P_d(X))\bigr)\right)
\longrightarrow
K_{*+m+1}\!\left(C_{Y,\infty}^*\!\bigl((P_d(X),\mathcal{A}(E))\bigr)\right),
\]
see \cite[Section 4]{SW07} for $K_0$-case and \cite[Section 7]{GLWZ24} for $K_1$-case.
\end{example}

Moreover, one can define the Bott map for localization algebras
\[
(\beta_L) :
K_*\left(C_{L,Y,\infty}^*\big(P_d(X)\big)\right)
\longrightarrow
K_{*+m+1}\left(C_{L,Y,\infty}^*\big(P_d(X),\mathcal{A}(E)\big)\right).
\]
For both constructions above, there exists the following commutative diagram:
\[
\begin{gathered}
K_*\!\left(C_{L,Y,\infty}^*\!\big(P_d(X)\big)\right)
\xrightarrow{\,(ev_{Y,\infty})_*\,}
K_*\!\left(C_{Y,\infty}^*\!\big(P_d(X)\big)\right)
\\
\downarrow\, {(\beta_L)}_* \qquad\qquad\qquad\qquad\qquad \downarrow\, \beta_* \\
K_*\!\left(C_{L,Y,\infty}^*\!\big(P_d(X),\mathcal{A}(E)\big)\right)
\xrightarrow{\,(ev_{Y,\infty})_*\,}
K_*\!\left(C_{Y,\infty}^*\!\big(P_d(X),\mathcal{A}(E)\big)\right).
\end{gathered}
\]

By using a cutting-and-pasting argument as in \cite{Yu00}, we can deduce the following lemma.
\begin{lemma}{\label{twist-eq}}
The Bott map for localization algebras 
$$\beta_L :
K_*\!\left(C_{L,Y,\infty}^*\!\big(P_d(X)\big)\right)
\longrightarrow
K_*\!\left(C_{L,Y,\infty}^*\!\big(P_d(X),\mathcal{A}(E)\big)\right)$$
is an isomorphism.\qed
\end{lemma}

By Theorem \ref{twist} and Lemma \ref{twist-eq}, and the following commutative diagram
\begin{center}
\begin{tikzcd}
K_{*}(C_{L,Y,\infty}^{*}(P_{d}(X))) \arrow[r, "(\beta_{L})_{*}"] \arrow[d, "(ev_{Y,\infty})_*"'] & 
K_{*+m+1}(C_{L,Y,\infty}^{*}(P_{d}(X),\mathcal{A}(E)))\arrow[d, "(ev_{Y,\infty}^{\mathcal{A}})_{*}"'] \\
K_{*}(C_{Y,\infty}^{*}(P_{d}(X))) \arrow[r, "\beta_{*}"] & 
K_{*+m+1}(C_{Y,\infty}^{*}(P_{d}(X),\mathcal{A}(E))),
\end{tikzcd}
\end{center}
we obtain the following consequence.

\begin{theorem}\label{thm: RCNC for RFCE}
Let $(X,d)$ be a metric space and let $Y\subseteq X$ be closed. 
If $X$ admits a fibred coarse embedding into either an $\ell^p$-space (with $1\le p<\infty$) or a Hadamard manifold \emph{relative to $Y$}, then the relative coarse Novikov conjecture holds for the pair $(X,Y)$.\qed
\end{theorem}

Combining \Cref{fcebdry} and \Cref{thm: RCNC for RFCE}, we deduce the following corollary.

\begin{corollary}\label{cor:relative-cn-product-fce}
Let $X=X_1\times\cdots\times X_n$ be a product metric space.
Assume that for each $i=1,\dots,n$, the space $X_i$ admits a fibred coarse embedding relative to $Y_i$
into either an $\ell^{p_i}$-space, with $1\le p_i<\infty$, or a Hadamard manifold.
Then the relative coarse Novikov conjecture holds for the pair
$(X,F_1)$, where $F_1$ is the $1$-boundary of $X$.\qed
\end{corollary}

\section{Iterated higher index theory at infinity}

In this section, we shall deduce the global coarse Novikov conjecture for $X$ from the relative statement for $(X, F)$, where the boundary $F = \bigcup_{i=1}^N X_i$ is a finite union of thickened faces. 

\subsection{Iterated higher index theory}
In previous approaches (in \cite{GWZ25}), the standard method proceeds by reducing the coarse Novikov conjecture for the pair $(X, F)$ to the coarse Novikov conjecture for the boundary $F$ itself. However, this argument strictly relies on the \emph{sparseness} of the spaces involved. If the product space contains any non-sparse factors, this standard reduction method may fail.

To bypass this obstacle, we develop a new reduction strategy. Rather than dealing with the entire boundary $F$ at once, our strategy is to decompose $F$ into its constituent faces $X_1, \dots, X_N$ and quotient them out \emph{one at a time}. This multi-stage reduction process naturally gives rise to a sequence of new algebras. To precisely describe this step-by-step reduction process, we introduce a sequence of new algebras. Because the reduction procedure itself is carried out iteratively, we call these new algebras the \emph{iterated relative Roe algebras} (and correspondingly, \emph{iterated relative localization algebras}). 

Recall that the ghostly ideal supported near a closed set $Y$ is
\[
I_G(Y)\;:=\;\Bigl\{\,T\in C^*(X)\ \Big|\ \forall \varepsilon>0\ \exists R>0\text{ such that }
\operatorname{supp}_\varepsilon(T)\subseteq \Pen(Y,R)\times \Pen(Y,R)\Bigr\},
\]
where $\operatorname{supp}_\varepsilon(T):=\{(x,y)\in X\times X:\ \|T(x,y)\|\ge\varepsilon\}$, this definition was originally introduced by Wang and Zhang in \cite{WZ23}.

\begin{definition}[Iterated relative Roe algebra]\label{def:iter-ideals}
Let $X$ be a proper metric space with bounded geometry, and let
$Y_1,\dots,Y_n\subseteq X$ be closed subsets. Let $I_{Y}$ denote the ghostly ideal
supported near $Y$.

Define the \emph{level-\(k\) iterated ghostly ideal} by
\[
I^{(k)}_{Y_1,\dots,Y_k}
\;:=\;
\overline{I_{Y_1}+I_{Y_2}+\cdots+I_{Y_k}}.
\]
The corresponding \emph{level-\(k\) iterated relative Roe algebra} is then defined by
\[
C^*_{\infty,Y_1,\dots,Y_k}(X)
\;:=\;
C^*(X)\big/ I^{(k)}_{Y_1,\dots,Y_k}.
\]
\end{definition}

\begin{remark}
When \(k=0\), this construction reduces to the standard Roe algebra $C^*(X)$. When \(k=1\), it is precisely the relative Roe algebra for $(X, Y_1)$ introduced in \cite{GWZ25}. 
In what follows, we will also refer to the level-$k$ iterated Roe algebra as the relative Roe algebra for $(X; Y_1, \dots, Y_k)$. For the sake of brevity, we will often omit the modifier ``level-$k$'' when no ambiguity arises.
\end{remark}

\begin{definition}\label{def:level-k-loc-alg}
\label{def:level-k-loc-alg}
Define the corresponding localization ideal by
\[
C_L^*\!\left(X;I^{(k)}_{Y_1,\dots,Y_k}\right)
:=
\left\{
g\in C_L^*(X):
g(t)\in I^{(k)}_{Y_1,\dots,Y_k}
\text{ for every }t\ge 0
\right\}.
\]
The \emph{level-\(k\) iterated relative localization algebra} is defined by
\[
C_{L,\infty,Y_1,\dots,Y_k}^*(X)
:=
C_L^*(X)\Big/
C_L^*\!\left(X;I^{(k)}_{Y_1,\dots,Y_k}\right).
\]
Thus, there is a short exact sequence
\[
0
\longrightarrow
C_L^*\!\left(X;I^{(k)}_{Y_1,\dots,Y_k}\right)
\longrightarrow
C_L^*(X)
\longrightarrow
C_{L,\infty,Y_1,\dots,Y_k}^*(X)
\longrightarrow 0.
\]
\end{definition}

\begin{definition}[Iterated relative $K$-homology]\label{def:iter-rel-K}
Let $X$ be a proper metric space with bounded geometry, and let $Y_1,\dots,Y_n \subseteq X$ be closed subsets. We define the \emph{iterated relative $K$-homology} of $X$ with respect to $Y_1,\dots,Y_n$ to be
\[ K_*^{\infty}(X;Y_1,\dots,Y_{k+1}) := K_*(C_{L,\infty,Y_1,\dots,Y_k}^*(X)).\]
 \end{definition}

When $k=1$, $K_*^{\infty}(X;Y_1,\dots,Y_{k+1})$ coincides with the relative $K$-homology group at infinity $K_*^{\infty}(X; Y_1)$ as in \cite{GWZ25}.

There is a canonical evaluation-at-zero map
        \[
          \mathrm{ev}:\; C_{L,\infty,Y_1,\dots,Y_k}^*(X)
          \longrightarrow C_{\infty,Y_1,\dots,Y_k}^*(X),
          \qquad g\mapsto g(0),
        \]
        which is well-defined and norm-decreasing.  Passing to \(K\)-theory and Rips complexes yields the level-\(k\) assembly map.

\begin{definition}[Iterated relative coarse Baum--Connes / Novikov conjectures]
Fix \(k\in\{1,\dots,n\}\).

\begin{enumerate}
  \item (\textbf{Level \(k\) iterated relative coarse Baum--Connes conjecture}).
  We say that the \emph{iterated relative coarse Baum--Connes conjecture} (or iterated relative CBC)
  for the tuple \((X;Y_1,\dots,Y_k)\) holds if the level-\(k\) assembly map
  \[
  \mu^{(k)}_{Y_1,\dots,Y_k,\infty}: \lim_{d\to\infty}K_*( C_{L,\infty,Y_1,\dots,Y_k}^*(P_d(X)))
          \longrightarrow K_*(C_{\infty,Y_1,\dots,Y_k}^*(X))
  \]
  is an isomorphism.

  \item (\textbf{Level \(k\) iterated relative coarse Novikov conjecture}).
  We say that the \emph{iterated relative coarse Novikov conjecture} (or iterated relative CNC)
  for \((X;Y_1,\dots,Y_k)\) holds if the level-\(k\) assembly map
  \[
  \mu^{(k)}_{Y_1,\dots,Y_k,\infty}: \lim_{d\to\infty}K_*( C_{L,\infty,Y_1,\dots,Y_k}^*(P_d(X)))
          \longrightarrow K_*(C_{\infty,Y_1,\dots,Y_k}^*(X))
  \]
  is injective.
\end{enumerate}
\end{definition}

Unwinding the notation, when \(k=1\) this coincides with the relative index map introduced in \cite{GWZ25},
\[
\mathrm{Ind}_{Y_1,\infty} \;:\; K_*^{\infty}(X;Y_1)\longrightarrow K_*\!\bigl(C_{\infty,Y_1}^*(X)\bigr).
\]

The following consequence shows that, in the finite case, the iterated relative theories agree with the corresponding single relative theory for the union. 
More precisely, the iterated relative \(K\)-homology, the iterated relative Roe algebra, and the iterated relative localization algebra all coincide with their relative counterparts associated to \(Y_1\cup\cdots\cup Y_k\).

\begin{proposition}\label{prop:sum-ideal-implies-equality}
Let \(X\) be a proper metric space and let \(Y_1,\dots,Y_n\subseteq X\) be closed subsets.
Write \(Y:=\bigcup_{j=1}^n Y_j\).
Then there exists a canonical isomorphism
\[
\Psi:\; C^*_{\infty,Y}(X)\longrightarrow C^*_{\infty,Y_1,\dots,Y_n}(X).
\]
Consequently,
\[
K_*\bigl(C^*_{\infty,Y}(X)\bigr)\ \cong\ K_*\bigl(C^*_{\infty,Y_1,\dots,Y_n}(X)\bigr).
\]
\end{proposition}

\begin{proof}
By definition of the iterated relative Roe algebra (via successive quotients), we have
\[
C^*_{\infty,Y_1,\dots,Y_n}(X)
=\frac{C^*(X)}{I^{(n)}_{Y_1,\dots,Y_n}},
\qquad
I^{(n)}_{Y_1,\dots,Y_n}
:=\overline{I_{Y_1}+\cdots+I_{Y_n}}.
\]
On the other hand,
\[
C^*_{\infty,Y}(X)=\frac{C^*(X)}{I_Y}.
\]
Thus it suffices to show
\[
I_Y=\overline{I_{Y_1}+\cdots+I_{Y_n}}.
\]

The inclusion
\[
\overline{I_{Y_1}+\cdots+I_{Y_n}}\subseteq I_Y
\]
is immediate from \(Y_i\subseteq Y:=\bigcup_{i=1}^n Y_i\). Conversely, by density it suffices to consider a finite-propagation operator \(T\in I_Y\); partition \(X=\bigsqcup_i A_i\) so that \(d(x,Y_i)=d(x,Y)\) for \(x\in A_i\), and write \(T=\sum_i\chi_{A_i}T\). Since each summand has its rows near \(Y_i\), finite propagation implies that its whole support lies in a bounded neighbourhood of \(Y_i\times Y_i\), hence \(\chi_{A_i}T\in I_{Y_i}\), and therefore
\[
I_Y\subseteq \overline{I_{Y_1}+\cdots+I_{Y_n}}.
\]

\end{proof}

\begin{corollary}\label{prop:iterK-equals-unionK}
Let $X$ be a proper metric space with bounded geometry, and let $Y_1,\dots,Y_n\subseteq X$ be closed subsets.
Put $Y:=\bigcup_{j=1}^n Y_j$. Then there is a canonical isomorphism
\[
K_*^{\infty}(X;Y_1,\dots,Y_n)\ \cong\ K_*^{\infty}(X,Y).
\]
As a result, the following diagram commutes\[
  \begin{tikzcd}
    \varinjlim_d K_*^{\infty}(P_d(X);Y)
      \ar[r,"\mu_{Y,\infty}"] \ar[d] & K_*\!\bigl(C_{\infty,Y}^*(X)\bigr) \ar[d, "(\Psi)_*"] \\
    \varinjlim_d K_*^{\infty}(P_d(X);Y_1,\dots,Y_n)
      \ar[r,"\mu^{(k)}_{Y_1,\dots,Y_k,\infty}"] & K_*\!\bigl(C_{\infty,Y_1,\dots,Y_k}^*(X)\bigr).
  \end{tikzcd}
  \]
\end{corollary}

\begin{proof}
The statements are all natural consequences of the functoriality of the localization and Roe constructions
and of the canonical quotient maps that define the iterated relative algebras.

For each $R>0$ there are canonical maps
\[
C_{L,\infty}^*(X\setminus Y_R)\longrightarrow C_{L,\infty}^*(U_R;Y_1,\dots,Y_{k-1}),
\] and likewise there is the natural quotient
\[
C_{\infty,Y}^*(X)\longrightarrow C_{\infty,Y_1,\dots,Y_k}^*(X).
\]
Passing to $K$-theory and then to the colimit over $R$ yields the commuting square displayed above;
this shows that the iterated index map is the composition of the union-relative index map with the
induced map on the quotient.
\end{proof}

\begin{remark}\label{rem:countable-ghost}
As a side note, the behavior of ghostly ideals changes when dealing with a countable family of closed subsets $\{Y_k\}_{k\in\mathbb N}$. While our finite-iteration result guarantees that $I_G^{(N)}(Y_1,\dots,Y_N) = I_G(\bigcup_{k=1}^N Y_k)$, for a countable family we define the countably iterated ghostly ideal by
\[
I_G^{(\omega)}(Y_1,Y_2,\dots)\;=\;
\overline{I_G(Y_1)+I_G(Y_2)+\cdots}:=\left\langle I_G(Y_1), I_G(Y_2),\cdots\right\rangle.
\]
In general, we only have the inclusion $I_G^{(\omega)}(Y_1,Y_2,\dots) \subseteq I_G(\bigcup_{k\ge 1} Y_k)$, and this inclusion can be strict. 

To see this strict inclusion, consider $X=\mathbb Z^2$ equipped with the Euclidean metric, and the rays $Y_k:=\{(kn,n)\in\mathbb Z^2 \mid n>0\}$ for each $k\ge 1$. Let $A:=\{(m,0)\in\mathbb Z^2 \mid m>0\}$ be the positive $x$-axis. Notice that $A$ is contained in the $1$-neighborhood of the infinite union $\bigcup_{k\ge 1}Y_k$, because $(m,1) \in Y_m$ is at distance $1$ from $(m,0)$. By definition, any element in $I_G^{(\omega)}(Y_1,Y_2,\dots)$ can be approximated by a finite sum, which implies it belongs to $I_G(\bigcup_{k=1}^n Y_k)$ for some integer $n$. However, for any fixed $n$, $A$ eventually leaves every bounded neighborhood of this finite union $\bigcup_{k=1}^n Y_k$. Consequently, an operator supported on $A$ belongs to $I_G(\bigcup_{k\ge 1} Y_k)$ but cannot be approximated by elements in $I_G^{(\omega)}(Y_1,Y_2,\dots)$.
\end{remark}

\subsection{Reduction}
Let us recall a basic dichotomy for bounded geometry proper metric spaces.

\begin{definition}[{\cite[Def.~5.4]{GLWZ24}}]\label{def:icc}
A bounded geometry proper metric space $X$ is said to have an \emph{infinite coarse component}
if there exists $R>0$ such that the Rips complex $P_R(X)$ has an unbounded connected component.
\end{definition}
Due to \cite[Lemma~5.5]{GLWZ24}, a bounded geometry proper metric space \(X\) either has an \emph{infinite coarse component},
or else \(X\) is a coarse disjoint union of a sequence of finite metric spaces.
In this paper, we call \(X\) a \emph{continuous space} in the former situation, and a \emph{sparse space} in the latter.

\begin{theorem}\label{thm:iterated-reduction}
Let \(N\ge 2\), and let \(X_1,\dots,X_N\) be proper metric spaces with bounded geometry.
Let
\[
X:=\prod_{i=1}^N X_i
\]
be the product space equipped with the \(\ell^2\)-product metric, and $F^{(n)}=\prod_{i\ne n}X_i\times\{o_n\}$ the $n$-th face, where $o_n\in X_n$ is the based point. Assume that the iterated coarse Novikov conjecture holds for $(F^{(n)}; F^{(1)}\cap F^{(n)},\dots,F^{(n-1)}\cap F^{(n)})$ and  the iterated relative coarse Novikov conjecture holds for
\[
(X,\ F^{(1)},\ F^{(2)},\ \dots,\ F^{(n)}).
\]
Then the iterated relative coarse Novikov conjecture also holds for
\[
(X,\ F^{(1)},\ F^{(2)},\ \dots,\ F^{(n-1)}).
\]
\end{theorem}

Since any space is either sparse or continuous, we shall prove this theorem in the following two cases.

\subsubsection{When $X_n$ is continuous}

To prove \Cref{thm:iterated-reduction} in this case, we first establish a geometric observation regarding the intersection of faces. Recall from Section 3.1 that $F_M^{(i)}$ denotes the $i$-th face of thickness $M$.

\begin{lemma}\label{lem:face-intersection}
Let $X = \prod_{i=1}^N X_i$, and let $F_M^{(i)}$ denote the $i$-th face of thickness $M$. Let $F^{(n)} \times \R_+$ denote the subset of $X$ where the $n$-th coordinate is restricted to a ray $\R_+ \subseteq X_n$. Then
\[
\left( \bigcup_{i=1}^{n-1} F_M^{(i)} \right) \cap \left(F^{(n)}\times \R_+ \right) = U_{<n} \times \R_+ \times X_{>n},
\]
where $U_{<n} \subseteq \prod_{i=1}^{n-1} X_i$ is the projection of $\bigcup_{i=1}^{n-1} F_M^{(i)}$ onto the first $n-1$ coordinates.
\end{lemma}

\begin{proof}
We decompose the product space as $X = X_{<n} \times X_n \times X_{>n}$, where $X_{<n} = \prod_{i=1}^{n-1} X_i$ and $X_{>n} = \prod_{i=n+1}^N X_i$. 

By definition, the condition for a point to lie in the first $n-1$ thickened faces imposes restrictions solely on the coordinates in $X_{<n}$. Thus, we can write $\bigcup_{i=1}^{n-1} F_M^{(i)} = U_{<n} \times X_n \times X_{>n}$ for some subset $U_{<n} \subseteq X_{<n}$. On the other hand, the set $F^{(n)} \times \R_+$ restricts only the $n$-th coordinate to $\R_+$, which means $F^{(n)}\times \R_+ = X_{<n} \times \R_+ \times X_{>n}$. 

Since the Cartesian product distributes over intersections, taking the intersection of these two sets yields:
\begin{align*}
\left( \bigcup_{i=1}^{n-1} F_M^{(i)} \right) \cap \left(F^{(n)} \times \R_+\right) 
&= (U_{<n} \times X_n \times X_{>n}) \cap (X_{<n} \times \R_+ \times X_{>n}) \\
&= (U_{<n} \cap X_{<n}) \times (X_n \cap \R_+) \times (X_{>n} \cap X_{>n}) \\
&= U_{<n} \times \R_+ \times X_{>n}.
\end{align*}
This completes the proof.
\end{proof}

\begin{lemma}\label{lem:direct-limit-zero-maps}
Assume that $X_n$ is continuous. Then the canonical inclusion
\[
\lim_{d\to\infty} \frac{C_L^*\left(P_d(X); I_{F^{(n)}}\right)}{C_L^*\left(P_d(X); I_{F^{(n)}}\right)\cap C_L^*\left(P_d(X); I^{(n-1)}_{F^{(1)},\dots,F^{(n-1)}}\right)}
\longrightarrow 
\lim_{d\to\infty} \frac{C_L^*\left(P_d(X)\right)}{C_L^*\left(P_d(X); I^{(n-1)}_{F^{(1)},\dots,F^{(n-1)}}\right)}
\]
induces the zero map on $K$-theory.
\end{lemma}

\begin{proof}
Since the space $X_n$ is continuous, for a sufficiently large scale $d$, we can find a geodesic ray $\R_+ \subseteq P_d(X_n)$. We may therefore identify $P_d(F^{(n)}) \times \R_+$ as a subspace of $P_d(X)$.

Recall that the ideal $I_{F^{(n)}}$ in the localization algebra corresponds to the subspace $F^{(n)}$. By the definition of localization algebras associated to subspaces, we can express the $K$-theory of the ideal as a direct limit over the thickened faces:
\[
K_*\left(C_L^*\left(P_d(X); I_{F^{(n)}}\right)\right) \cong \lim_{M\to\infty} K_*\left(C_L^*\left(P_d(F_M^{(n)})\right)\right).
\]
For $d$ sufficiently large, the Rips complex of the thickened face $P_d(F_M^{(n)})$ is homotopy equivalent to $P_d(F^{(n)})$. Consequently, when passing to the limit as $d \to \infty$, the induced map on $K$-theory factors through the $K$-theory of $C_L^*(P_d(F^{(n)}) \times \R_+)$ modulo the intersection with the ideal of the previous $n-1$ faces. 

To simplify notation, let $J_{<n} = I^{(n-1)}_{F^{(1)},\dots,F^{(n-1)}}$. The inclusion thus factors through the quotient algebra:
\[ 
\frac{C_L^*(P_d(F^{(n)}) \times \R_+)}{C_L^*(P_d(F^{(n)}) \times \R_+) \cap C_L^*(P_d(X); J_{<n})}. 
\]
It is a standard result that the localization algebra $C_L^*(P_d(F^{(n)}) \times \R_+)$ is flasque. Therefore, we have $K_*\left(C_L^*(P_d(F^{(n)}) \times \R_+)\right) = 0$. 

By the six-term exact sequence in $K$-theory, the quotient algebra will also have vanishing $K$-theory if the intersection ideal 
\[ 
\mathcal{I}_{\cap} := C_L^*(P_d(F^{(n)}) \times \R_+) \cap C_L^*(P_d(X); J_{<n}) 
\]
has vanishing $K$-theory.

We now analyze $\mathcal{I}_{\cap}$ explicitly. The ideal $J_{<n}$ is generated by the geometric subspaces $F^{(1)}, \dots, F^{(n-1)}$, which means its associated localization algebra at scale $d$ can be expressed as the inductive limit over the thickened faces:
\[
C_L^*\left(P_d(X); J_{<n}\right) = \lim_{M\to\infty} C_L^*\left(P_d\left(\bigcup_{i=1}^{n-1} F_M^{(i)}\right)\right).
\]
By definition, the intersection of localization algebras associated to geometric subspaces coincides with the localization algebra of their geometric intersection. Therefore, we can write the intersection ideal as:
\begin{align*}
\mathcal{I}_{\cap} &= C_L^*\left(P_d(F^{(n)}) \times \R_+\right) \cap \lim_{M\to\infty} C_L^*\left(P_d\left(\bigcup_{i=1}^{n-1} F_M^{(i)}\right)\right) \\
&= \lim_{M\to\infty} C_L^*\left(  \left(P_d\left(F^{(n)} \right)\times \R_+\right) \cap P_d\left(\bigcup_{i=1}^{n-1} F_M^{(i)} \right) \right).
\end{align*}
Applying \Cref{lem:face-intersection}, this geometric intersection cleanly separates the $\R_+$ factor:
\[ 
\left( F^{(n)} \times \R_+ \right) \cap \left( \bigcup_{i=1}^{n-1} F_M^{(i)} \right) = \left( F^{(n)} \cap \bigcup_{i=1}^{n-1} F_M^{(i)} \right) \times \R_+. 
\]
Substituting this back, we obtain:
\[
\mathcal{I}_{\cap} = \lim_{M\to\infty} C_L^*\left( P_d\left( F^{(n)} \cap \bigcup_{i=1}^{n-1} F_M^{(i)} \right) \times \R_+ \right).
\]
Because this intersection explicitly retains the Cartesian factor $\R_+$, the corresponding localization algebra $\mathcal{I}_{\cap}$ is again flasque. 

This implies $K_*(\mathcal{I}_{\cap}) = 0$. Since both the flasque algebra and its ideal have zero $K$-theory, their quotient has zero $K$-theory. Consequently, as $d \to \infty$, the canonical inclusion map factors through an algebra with vanishing $K$-theory, meaning it induces the zero map on $K$-theory.
\end{proof}

\begin{proof}[Proof of \Cref{thm:iterated-reduction}; when $X_n$ is continuous]
By definition of the iterated relative localization algebras, for each fixed scale $d>0$, we have a short exact sequence of localization algebras:
\[
\begin{tikzcd}
0 \arrow[r] 
& \ker(\pi_L) \arrow[r] 
& C_{L,\infty,F^{(1)},\dots, F^{(n-1)}}^{*}(P_{d}(X)) \arrow[r,"{\pi_L}"] 
& C_{L,\infty,F^{(1)},\dots,F^{(n)}}^{*}(P_{d}(X)) \arrow[r] 
& 0,
\end{tikzcd}
\]
where the kernel $\ker(\pi_L)$ is precisely the quotient algebra
\[
\ker(\pi_L) \cong \frac{C_L^*\left(P_d(X); I_{F^{(n)}}\right)}{C_L^*\left(P_d(X); I_{F^{(n)}}\right)\cap C_L^*\left(P_d(X); I^{(n-1)}_{F^{(1)},\dots,F^{(n-1)}}\right)}
\]
analyzed in \Cref{lem:direct-limit-zero-maps}. Passing to $K$-theory, we obtain the following commutative diagram relating the localization $K$-theory groups to the Roe algebras at infinity:
\[
\begin{tikzcd}
K_*\left(C_{L,\infty,F^{(1)},\dots, F^{(n-1)}}^{*}(P_{d}(X))\right) \arrow[r,"{(\pi_L)_*}"] \arrow[d,"{\mu^{(n-1)}_{F^{(1)},\dots, F^{(n-1)},\infty}}"'] 
& K_*\left(C_{L,\infty,F^{(1)},\dots,F^{(n)}}^{*}(P_{d}(X))\right) \arrow[d,"{\mu^{(n)}_{F^{(1)},\dots, F^{(n)},\infty}}"]  \\ 
K_*\left(C_{\infty,F^{(1)},\dots, F^{(n-1)}}^{*}(P_{d}(X))\right) \arrow[r,"{\pi_*}"] 
& K_*\left(C_{\infty,F^{(1)},\dots,F^{(n)}}^{*}(P_{d}(X))\right)
\end{tikzcd}
\]

Now, taking the direct limit as $d \to \infty$, \Cref{lem:direct-limit-zero-maps} proves that the map induced by the inclusion of the kernel $\ker(\pi_L)$ into $C_{L,\infty,F^{(1)},\dots, F^{(n-1)}}^{*}(P_{d}(X))$ is the zero map on $K$-theory. 
By exactness, this vanishing implies that the induced map on $K$-theory $(\pi_L)_*$ is injective.

By a standard diagram chase argument, the injectivity of $\mu^{(n-1)}_{F^{(1)},\dots, F^{(n-1)},\infty}$ follows directly from the iterated relative coarse Novikov conjecture for $(X; F^{(1)}, \dots, F^{(n)})$. This finishes the proof.
\end{proof}

\subsubsection{When $X_n$ is sparse}

\begin{lemma}\label{lem:sparse-face-injection}
Assume that $X_n$ is sparse. Then the canonical inclusion
\[
\frac{C^*\left(X; F^{(n)}\right)}{C^*\left(X; {F^{(n)}}\right)\cap I^{(n-1)}_{F^{(1)},\dots,F^{(n-1)}}}
\longrightarrow
C^*_{\infty,F^{(1)},\dots,F^{(n-1)}}\left(X\right)
\]
induces an injection on $K$-theory.
\end{lemma}

\begin{proof}
For simplicity of notation, let $I_{<n} := I^{(n-1)}_{F^{(1)},\dots,F^{(n-1)}}$ denote the ghostly ideal generated by the first $n-1$ faces. 

Since $X_n$ is sparse, it can be expressed as a coarse disjoint union of finite metric spaces $X_n = \bigsqcup_{k=1}^\infty Z_k$. Consequently, the product space $X$ decomposes into a coarse disjoint union $\bigsqcup_{k=1}^{\infty}X_{(k)}$, where $X_{(k)} := X_{<n} \times Z_k \times X_{>n}$.

Let $C^*(X')$ denote the Roe algebra of $X$ equipped with the \emph{separated disjoint union metric} along the $n$-th coordinate. That is, for any $x \in X_{(k)}$ and $x' \in X_{(k')}$, we define
\[
d_{X'}(x,x') = \begin{cases} 
\infty, & \text{if } k \ne k', \\ 
d_X(x,x'), & \text{if } k = k'.
\end{cases}
\]
Since any operator with finite propagation in $X'$ naturally has finite propagation in $X$, we have $\mathbb{C}[X'] \subseteq \mathbb{C}[X]$. Thus, $C^*(X')$ canonically embeds as a $C^*$-subalgebra of $C^*(X)$. We define the corresponding geometric ideal and the ghostly ideal in $C^*(X')$ by intersection:
\[
C^*(X'; F^{(n)}) := C^*(X; F^{(n)}) \cap C^*(X'), \quad \text{and} \quad I'_{<n} := I_{<n} \cap C^*(X').
\]
Let $Y_K = \bigsqcup_{k=1}^K X_{(k)}$, equipped with the sparse metric, and $Y'_K$ the same set but equipped with the separated metric. The geometric ideal $C^*(X'; F^{(n)})$ is exactly the inductive limit of $C^*(Y'_K)$ as $K \to \infty$.

We construct the following commutative diagram on $K$-theory induced by the inclusion $C^*(X') \hookrightarrow C^*(X)$:
\[
\begin{tikzcd}
0 \arrow[r] & \lim\limits_{K\to\infty} K_*\left( \frac{C^*(Y'_K)}{C^*(Y'_K) \cap I'_{<n}} \right) \arrow[r] \arrow[d] & K_*\left( \frac{C^*(X')}{I'_{<n}} \right) \arrow[r, "p_*"] \arrow[d] & K_*\left( \frac{C^*(X')}{C^*(X'; F^{(n)}) + I'_{<n}} \right) \arrow[r] \arrow[d, "\cong"', "i_*"] & 0 \\
& K_*\left( \frac{C^*(X; F^{(n)})}{C^*(X; F^{(n)}) \cap I_{<n}} \right) \arrow[r, "\iota_*"] & K_*\left( \frac{C^*(X)}{I_{<n}} \right) \arrow[r, "q_*"] & K_*\left( \frac{C^*(X)}{C^*(X; F^{(n)}) + I_{<n}} \right) &
\end{tikzcd}
\]

We first verify the exactness of the top row. For any $K$, the spatial truncation map $T \mapsto \chi_{Y'_K} T \chi_{Y'_K}$ defines a bounded $*$-homomorphism on $C^*(X')$ because $Y'_K$ is infinitely far from its complement under the separated metric. Furthermore, since the conditions defining $I'_{<n}$ depend solely on the first $n-1$ coordinates, this truncation simply sent $I'_{<n}$ to $I'_{<n}\cap C*(Y'_K)$, and thus descends to a well-defined left inverse for the quotient inclusion $\frac{C^*(Y'_K)}{C^*(Y'_K) \cap I'_{<n}} \hookrightarrow \frac{C^*(X')}{I'_{<n}}$. Taking the direct limit over $K$, the top-left horizontal map is injective, which implies that the map $p_*$ is surjective by the exactness of the sequence for $X'$.

Next, we prove that the rightmost vertical map $i_*$ is an isomorphism. The sparse condition implies that $d_{X_n}(Z_k, Z_{k'}) \to \infty$ as $\max(k, k') \to \infty$ for $k \ne k'$. Thus, for any propagation $R > 0$, there exists a sufficiently large $K$ such that the distance between any two distinct components $Z_k, Z_{k'}$ for $k, k' > K$ is strictly greater than $R$. Therefore, any operator $T \in \mathbb{C}[X]$ with propagation $R$ cannot contain non-zero matrix entries connecting $X_{(k)}$ and $X_{(k')}$ outside of $Y_K$. This means $T$ can be decomposed as $T = T' + T_K$, where $T'$ is block-diagonal with respect to the components $(X_{(k)})_{k \in \mathbb{N}}$, and $T_K$ is supported entirely on $Y_K \times Y_K$.  Hence, $T' \in \mathbb{C}[X']$. Since $Y_K$ is uniformly bounded in the $n$-th coordinate, $T_K \in \mathbb{C}[X; F^{(n)}]$. This provides the decomposition $\mathbb{C}[X] = \mathbb{C}[X'] + \mathbb{C}[X; F^{(n)}]$, which yields a canonical algebraic isomorphism:
\[
\frac{\mathbb{C}[X]}{\mathbb{C}[X; F^{(n)}]} \cong \frac{\mathbb{C}[X']}{\mathbb{C}[X'] \cap \mathbb{C}[X; F^{(n)}]} = \frac{\mathbb{C}[X']}{\mathbb{C}[X'; F^{(n)}]}.
\]
Passing to the $C^*$-completion yields $\frac{C^*(X)}{C^*(X; F^{(n)})} \cong \frac{C^*(X')}{C^*(X'; F^{(n)})}$. Quotienting further by the ideal $I_{<n}$ preserves this canonical identification, which means $\frac{C^*(X)}{C^*(X; F^{(n)}) + I_{<n}} \cong \frac{C^*(X')}{C^*(X'; F^{(n)}) + I'_{<n}}$. Thus, the map $i_*$ is an isomorphism.

Since $i_* \circ p_*$ is surjective, a standard diagram chase ensures that the bottom horizontal map $q_*$ must be surjective. By the six-term exact sequence, this is equivalent to the inclusion $\iota_*$ being injective, which completes the proof.
\end{proof}

\begin{proof}[Proof of \Cref{thm:iterated-reduction}; when $X_n$ is sparse]
Using the notation in the proof of \Cref{lem:sparse-face-injection}, we have the following commutative diagram:
\[
\begin{tikzcd}
0 \arrow[d] & 0 \arrow[d] \\
\displaystyle \lim_{d\to\infty}
K_*\left(\frac{C^*_L(P_d(X); F^{(n)})}{C^*_L(P_d(X); F^{(n)}) \cap C^*_L(P_d(X); I_{<n})} \right)
  \arrow[r,"\mu"] \arrow[d]
& K_*\left( \frac{C^*(X; F^{(n)})}{C^*(X; F^{(n)}) \cap I_{<n}} \right)  \arrow[d,"\iota_*"] \\
\displaystyle \lim_{d\to\infty} K_*\left(C^*_{L,\infty,F^{(1)},\dots,F^{(n-1)}}\left(P_d(X)\right)\right)
  \arrow[r,"\mu_{F^{(1)},\dots,F^{(n-1)},\infty}^{(n-1)}"] \arrow[d]
& K_*\left(C^*_{\infty,F^{(1)},\dots,F^{(n-1)}}\left(X\right)\right) \arrow[d] \\
\displaystyle \lim_{d\to\infty}
  K_*\left(C^*_{L,\infty,F^{(1)},\dots,F^{(n)}}\left(P_d(X)\right)\right)
  \arrow[r,"\mu_{F^{(1)},\dots,F^{(n)},\infty}^{(n)}"] \arrow[d]
& K_*\left(C^*_{\infty,F^{(1)},\dots,F^{(n)}}\left(X\right)\right)
  \\
0
\end{tikzcd}
\]
By \Cref{lem:sparse-face-injection}, the top right vertical map $\iota_*$ is injective. Note that the top horizontal map $\mu$ is naturally identified with the iterated relative assembly map for the tuple $(F^{(n)}; F^{(1)}\cap F^{(n)}, \dots, F^{(n-1)}\cap F^{(n)})$. By our assumption, both $\mu$ and the bottom horizontal map $\mu_{F^{(1)},\dots,F^{(n)},\infty}^{(n)}$ are injective. Since $\iota_*$ is also injective, a standard diagram chase immediately yields that the middle horizontal map $\mu_{F^{(1)},\dots,F^{(n-1)},\infty}^{(n-1)}$ is injective. This establishes the iterated relative coarse Novikov conjecture for $(X; F^{(1)}, \dots, F^{(n-1)})$ and completes the proof.
\end{proof}

\section{Main result for product spaces}

Our main theorem is about the higher index problem for product spaces, which is stated as follows:

\begin{theorem}\label{mainthm}
Let $N\ge 2$, and let $X_1,\dots,X_N$ be proper metric spaces with bounded geometry. 
Assume that for each $i$ the space $X_i$ admits a fibred coarse embedding into one of the following model targets:
a Hilbert space, a Hadamard manifold, or an $\ell^p$-space with $1\le p<\infty$.
Let $X := \prod_{i=1}^N X_i$ be the finite product equipped with the $\ell^2$-product metric, and let $F^{(i)}$ denote the $i$-th coordinate face. Then the following hold:
\begin{enumerate}
    \item[(1)] The \textbf{coarse Novikov conjecture holds for the finite product $X$}.
    \item[(2)] For any $1 \le k \le N$ and any distinct indices $i_1, \dots, i_k \in \{1, \dots, N\}$, the \textbf{iterated relative coarse Novikov conjecture holds for the tuple $(X; F^{(i_1)}, \dots, F^{(i_k)})$}.
\end{enumerate}
\end{theorem}

\begin{proof}
We prove the theorem by induction on the number of factors $N$.

For the base case $N=2$, let $X = X_1 \times X_2$. The two coordinate faces are $F^{(1)}$ and $F^{(2)}$. Note that $F^{(1)}=X_2$, and $F^{(2)}=X_1$. Since $X_1$ and $X_2$ admit fibred coarse embeddings into the prescribed model spaces, the coarse Novikov conjecture holds for themselves and the pairs $(X_1, o_1)$, $(X_2, o_2)$. 

By \Cref{cor:relative-cn-product-fce}, the relative coarse Novikov conjecture holds for the pair $(X, F^{(1)} \cup F^{(2)})$. Using \Cref{prop:iterK-equals-unionK}, the iterated relative coarse Novikov conjecture holds for the tuple $(X; F^{(1)}, F^{(2)})$. This establishes statement (2) for $k=2$. Since the relative coarse Novikov conjecture holds for $(X_1,o_1)$, applying \Cref{thm:iterated-reduction} to $F^{(2)}$, we deduce that the iterated relative coarse Novikov conjecture holds for $(X; F^{(1)})$. Applying the reduction theorem once more to $F^{(1)}$, we obtain the global coarse Novikov conjecture for $X$. This establishes both statements for $N=2$.

Now assume that $N \geq 3$, and suppose the theorem holds for all products of up to $N-1$ spaces satisfying the hypotheses. Let $X = \prod_{i=1}^N X_i$. 

Let $F = \bigcup_{i=1}^N F^{(i)}$ be the $1$-boundary of $X$. By \Cref{cor:relative-cn-product-fce}, $X$ admits a fibred coarse embedding relative to $F$, meaning the iterated relative coarse Novikov conjecture holds for the maximal tuple $(X; F^{(1)}, \dots, F^{(N)})$. This establishes (2) for $k=N$.

To prove (2) for $k < N$ and (1), we iteratively apply \Cref{thm:iterated-reduction} to peel off the faces one by one. A crucial observation is that any face $F^{(n)} \cong \prod_{j \ne n} X_j$ is itself a product of $N-1$ factors. Furthermore, the intersection of $F^{(n)}$ with any other face $F^{(i)}$ ($i \ne n$) corresponds to a face of $F^{(n)}$. 

By our inductive hypothesis applied to the $(N-1)$-product space $F^{(n)}$, the coarse Novikov conjecture holds for $F^{(n)}$. Additionally, the iterated relative coarse Novikov conjecture holds for any tuple of its faces. This satisfies the prerequisites of \Cref{thm:iterated-reduction}.

Thus, applying \Cref{thm:iterated-reduction} to the tuple $(X; F^{(1)}, \dots, F^{(N)})$, we can peel off $F^{(N)}$ to deduce that the iterated relative coarse Novikov conjecture holds for $(X; F^{(1)}, \dots, F^{(N-1)})$. Since the ordering of the faces is arbitrary, this holds for any sub-tuple of length $N-1$. Continuing this process inductively, we establish (2) for any $k \ge 1$. 

Finally, applying the reduction theorem to the single-face case $(X; F^{(1)})$, we conclude that the coarse Novikov conjecture holds for the global space $X$.
\end{proof}

Combining \cite{CWW13,SWW21,Yu05}, we obtain the following result.

\begin{corollary}[Products of box spaces and warped cones]
Let \(N,M\ge 1\).

For \(i=1,\dots,N\), let \(\Gamma_i\) be a finitely generated residually finite hyperbolic group, and
let \(\{\Gamma_{i,n}\}_{n\in\mathbb N}\) be a filtration of \(\Gamma_i\).
Write \(\Box_{\{\Gamma_{i,n}\}}\Gamma_i\) for the associated box space.

For \(j=1,\dots,M\), let \(\Lambda_j\) be a finitely generated discrete group acting on a compact manifold \(M_j\),
and assume that
\begin{enumerate}
  \item the action \(\Lambda_j\curvearrowright M_j\) is linearisable in an \(\ell^{p_j}\)-space and is free
  (or \(M_j\) contains a dense free orbit);
  \item \(\Lambda_j\) admits a proper affine isometric action on an \(\ell^{p_j}\)-space,
\end{enumerate}
for some \(1\le p_j<\infty\).
Let \(\mathcal O_{\Lambda_j}(M_j)\) be the corresponding warped cone.

Then the coarse Novikov conjecture holds for the mixed finite product
\[
X\ :=\ \Big(\prod_{i=1}^N \Box_{\{\Gamma_{i,n}\}}\Gamma_i\Big)\ \times\
\Big(\prod_{j=1}^M \mathcal O_{\Lambda_j}(M_j)\Big).
\]
\end{corollary}

\medskip

\noindent\textsc{(L.~Guo)} Shanghai Institute for Mathematics and Interdisciplinary Sciences, Shanghai 200433, P.\ R.\ China.\\
\textit{Email address:} \texttt{liang\_guo@fudan.edu.cn}

\medskip

\noindent\textsc{(Z.~Luo)} College of Data Science, Jiaxing University, Jiaxing, Zhejiang, P.\ R.\ China.
\textit{Email address:} \texttt{zluo@zjxu.edu.cn}

\medskip

\noindent\textsc{(Q.~Wang)} Research Center for Operator Algebras, School of Mathematical Sciences, East China Normal University, Shanghai 200241, P.\ R.\ China.\\
\textit{Email address:} \texttt{qwang@math.ecnu.edu.cn}

\end{document}